\documentclass[pdflatex,sn-mathphys-num]{sn-jnl}

\usepackage{graphicx}
\usepackage{multirow}
\usepackage{amsmath,amssymb,amsfonts}
\usepackage{amsthm}
\usepackage{mathrsfs}
\usepackage[title]{appendix}
\usepackage{xcolor}
\usepackage{textcomp}
\usepackage{manyfoot}
\usepackage{booktabs}
\usepackage{algorithm}
\usepackage{algorithmicx}
\usepackage{algpseudocode}
\usepackage{listings}

\usepackage{empheq} % Ez a csomag szükséges
\usepackage{hyperref}

\theoremstyle{thmstyleone}
\newtheorem{theorem}{Theorem}[section]
\newtheorem{lemma}{Lemma}[section]

\newtheorem{corollary}{Corollary}[section]

\theoremstyle{thmstyletwo}

\theoremstyle{thmstylethree}

\newtheorem{remark}{Remark}[section]
\newtheorem{openproblem}{Open Problem}[section]
\newtheorem{notation}{Notation}[section]

\begin{document}

\title[Qualitative Analysis of a Longitudinal T2DM Model]{Qualitative Properties and Stability Analysis of Autonomous and Non-Autonomous Longitudinal T2DM Models}

\author[1]{\fnm{V. E. S.} \sur{Szab\'o}}\email{bmesszabo@gmail.com}

\affil[1]{\orgdiv{Department of Analysis and Operations Research, Institute of Mathematics}, 
\orgname{Budapest University of Technology and Economics}, 
\orgaddress{\street{M\H uegyetem rkp. 3.}, \city{Budapest}, \postcode{H-1111}, \country{Hungary}}}

\abstract{In this paper, we investigate the qualitative properties, long-term dynamic behavior, 
and stability of a longitudinal mathematical model describing the progression of Type 2 Diabetes 
Mellitus (T2DM). 
We significantly extend and refine the mathematical analysis of the model recently proposed by 
De Gaetano et al. (Journal of Theoretical Biology, 2024).
Specifically, we analyze both the autonomous system with a fixed age parameter and the non-autonomous system 
incorporating explicit time-dependent aging dynamics under generalized conditions, providing rigorous proofs for 
global existence, uniqueness, positivity, asymptotic bounds, and stability properties.}

\maketitle

\section{Introduction}\label{sec:Intro}

Type 2 Diabetes Mellitus (T2DM) is a chronic metabolic disorder characterized by insulin 
resistance, impaired insulin secretion, and progressive pancreatic $\beta$-cell dysfunction. 
Mathematical modeling of glucose-insulin homeostasis plays a crucial role in understanding the 
long-term pathogenesis of diabetes. While classical minimal models focus on short-term dynamics 
(such as intravenous glucose tolerance tests operating on minute or hour scales), longitudinal 
models are required to capture the slow progression of metabolic degradation over years and 
decades.

Recently, De Gaetano et al. \cite{degaetano2024} proposed a comprehensive longitudinal ordinary 
differential equation (ODE) model describing the progression of T2DM. 

In order to simplify the notation of the biological constants while retaining the original model 
structure, we denote the positive system rate parameters by $k_1, k_2, \dots, k_{16} > 0$. 

In De Gaetano et al. \cite{degaetano2024} the following longitudinal model was proposed
\begin{empheq}[left=\empheqlbrace]{align}
    G'(t) &=k_1e^{- L(t)I(t)G(t)}-S(t)I(t)G(t)-R(G(t)), \label{eq:1} \\
    B'(t) &=k_3e^{-\lambda (A(t)-A_0)}-k_4G(t)B(t),  \label{eq:2} \\
    I'(t) &=-k_5I(t)+k_6 \frac{B(t)}{B_0} \frac{(G(t))^{\gamma}}{c_1+(G(t))^{\gamma}},  
                                                                       \label{eq:3} \\
    S'(t) &=k_7-k_8S(t)+k_9Y(t)-k_{10}F(t)S(t)-k_{11}(A(t)-A_0)S(t), \label{eq:4} \\
    L'(t) &=k_{12}-k_{13}L(t)+k_{14}Y(t)-k_{15}F(t)L(t)-k_{16}(A(t)-A_0)L(t), \label{eq:5} \\
    A'(t) &=1 \label{eq:6}
\end{empheq}
with the initial conditions
\begin{equation}
    G(t_0)=G_0, B(t_0)=B_0, I(t_0)=I_0, S(t_0)=S_0, L(t_0)=L_0, A(t_0)=A_0 \label{eq:init}
\end{equation}
where
\begin{equation}\label{eq:renal_R}
R(G) = \frac{s}{2}\left(\sqrt{4\rho_{GR}G(t)+(G(t)-G_R-\rho_{GR})^2}+(G(t)-G_R-\rho_{GR})   
\right).
\end{equation}
The initial values $G_0, B_0, I_0, S_0, L_0$ are positive numbers, $Y(t)=Y_0\geq 0$, 
$F(t)=F_0\geq 0$ for all $t\geq 0$, and $\lambda, \gamma, c_1, s$ are positive numbers. 
The state variables of the 
model represent the key physiological components of the glycemic control system: $G(t)$ denotes 
the representative (typical or fasting) plasma glucose concentration, $B(t)$ is the pancreatic 
$\beta$-cell mass, $I(t)$ is the plasma insulin concentration, $S(t)$ represents peripheral 
insulin sensitivity, and $L(t)$ corresponds to hepatic insulin sensitivity. The input functions are 
the non-negative functions $Y(t)$, the effect 
of physical activity and $F(t)$, the effect of food intake, both expressed
as kcal (respectively of average daily above-normal physical activity
and average daily excess food intake).  For full 
physiological details and block diagrams, see Table 1 and Fig. 1 in \cite{degaetano2024}. 
Without loss of generality we may assume that $t_0=0$. In the same paper the local stability for 
fixed age ($A(t)=A_0$) was investigated with the assumptions $F(t)=F_0$, $Y(t)=Y_0$. 

In this paper, we investigate the qualitative properties 
of the following 
more general 
 systems incorporating biological age evolution $A(t) = t + A_0$:
\begin{empheq}[left=\empheqlbrace]{align}
    G'(t) &=k_1e^{-\alpha_0 L(t)I(t)G(t)}-k_2S(t)I(t)G(t)-u(G(t)), \label{eq:non_auto_1} \\
    B'(t) &=k_3e^{-\lambda (A(t)-A_0)}-k_4G(t)B(t),  \label{eq:non_auto_2} \\
    I'(t) &=-k_5I(t)+k_6 \frac{B(t)}{B_0} \frac{|G(t)|^{\gamma}}{c_1+|G(t)|^{\gamma}},  
                                                                       \label{eq:non_auto_3} \\
    S'(t) &=k_7-k_8S(t)+k_9Y(t)-k_{10}F(t)S(t)-k_{11}(A(t)-A_0)S(t), \label{eq:non_auto_4} \\
    L'(t) &=k_{12}-k_{13}L(t)+k_{14}Y(t)-k_{15}F(t)L(t)-k_{16}(A(t)-A_0) L(t), \label{eq:non_auto_5} \\
    A'(t) &=1 \label{eq:non_auto_6}
\end{empheq}
alongside its counterpart, obtained by fixing the biological age at $A(t)= A_0$:
\begin{empheq}[left=\empheqlbrace]{align}
    G'(t) &=k_1e^{-\alpha_0 L(t)I(t)G(t)}-k_2S(t)I(t)G(t)-u(G(t)), \label{eq:auto_1} \\
    B'(t) &=k_3-k_4G(t)B(t),  \label{eq:auto_2} \\
    I'(t) &=-k_5I(t)+k_6 \frac{B(t)}{B_0} \frac{|G(t)|^{\gamma}}{c_1+|G(t)|^{\gamma}},  
                                                                       \label{eq:auto_3} \\
    S'(t) &=k_7-k_8S(t)+k_9Y(t)-k_{10}F(t)S(t), \label{eq:auto_4} \\
    L'(t) &=k_{12}-k_{13}L(t)+k_{14}Y(t)-k_{15}F(t)L(t), \label{eq:auto_5} 
\end{empheq}
subject to positive initial conditions
\begin{equation}\label{eq:initial_cond}
G(0)=G_0,\,B(0)=B_0,\, I(0)=I_0,\,S(0)=S_0,\,L(0)=L_0,
\end{equation}
where the non-smooth piecewise linear function $u(G)$ represents renal glucose excretion:
\begin{equation}\label{eq:renal_u}
u(G)=
\begin{cases}
0, & \text{if } G<G_u,\\
s(G-G_u), & \text{if } G\geq G_u;
\end{cases}
\end{equation}
and the system coefficients satisfy
\begin{equation}\label{eq:sys_coeff}
\alpha_0\geq 0,\, \gamma\geq 1,\, \lambda, c_1, s, G_u, G_0,\ldots,L_0,  k_1,\ldots, k_{16}>0,
\end{equation}
under general conditions 
\begin{equation}\label{eq:F_Y}
0\leq c_2\leq Y(t)\leq c_3,\quad 0\leq c_4 \leq F(t)\leq c_5 \quad (t\geq 0),
\end{equation}
and stability of the system $A(t)=A_0$ 
in the autonomous setting where $Y(t) \equiv Y_0$ and $F(t) \equiv F_0$.

\begin{remark}
It is worth noting an important technical distinction regarding the formulation of renal glucose 
excretion. In the model presented in \cite{degaetano2024}, instead of using the semi-classical 
piecewise linear function $u(G)$ defined in \eqref{eq:renal_u}, the authors utilized a smooth, 
twice-differentiable approximation $R(G)$ given by (see (8) in \cite{degaetano2024}):
As discussed in \cite{degaetano2024}, the substitution of $u(G)$ by $R(G)$ was motivated by 
analytical convenience, specifically to satisfy the differentiability assumptions required by the 
specific mathematical tools employed in their local analysis. The relationship and numerical 
approximation between $u(G)$ and $R(G)$ are illustrated in Fig. 2 of \cite{degaetano2024}. 

However, from both physical and physiological standpoints, a threshold-based piecewise linear 
function such as $u(G)$ provides a highly realistic mechanistical model of the sharp activation 
of renal glucose clearance once the physiological renal threshold $G_u$ is exceeded. A key 
feature of our approach is that our analytical proofs rely on a different theoretical framework 
that does not require continuous differentiability of the excretion term. Consequently, we are 
able to establish global qualitative properties and stability directly for the natural, non-
smooth piecewise linear formulation $u(G)$.
\end{remark}
\begin{remark}
    The formula \eqref{eq:3} is incorrect, because we do not know that $G(t)\geq 0$ so the 
    power $\gamma$ of it is problematic. So we should modify the formula for the following form 
\begin{equation*}
    I'(t) =-k_5I(t)+k_6 \frac{B(t)}{B_0} \frac{|G(t)|^{\gamma}}{c_1+|G(t)|^{\gamma}}.
\end{equation*}
\end{remark}
\begin{remark}\label{rem:frac_gamma}
The authors of \cite{degaetano2024} use Volpert's theorem, and they need the continuous 
differentiability of the function 
\begin{equation}
    h(x)=\frac{|x|^{\gamma}}{c+|x|^{\gamma}}\quad (c>0)
\end{equation}
which is satisfied only for $\gamma >1$.

If $0<\gamma<1$ then $h$ is not differentiable at $0$. If $\gamma=1$ then 
\begin{equation*}
    h'(0+)=\frac{1}{c_1}\neq -\frac{1}{c_1}=h'(0-),
\end{equation*}
hence $h$ is not differentiable at $0$. If $\gamma>1$ then 
\begin{equation*}
    h'(0+)=h'(0-)=0 \implies h'(0)=0,
\end{equation*}
and
\begin{equation*}
    \lim_{x\to 0+} h'(x)=\lim_{x\to 0-} h'(x)=0=h'(0)
\end{equation*}
If $x\neq 0$ then $h'(x)$ is continuous for $\gamma >0$. 

Instead of Volpert's theorem we use contraction principle and we will be able to extend 
the results for $\gamma\geq 1$ (see \autoref{sec:Exist_and_unique} and \autoref{sec:Positivity}).   
\end{remark}
This motivates the following 
\begin{openproblem}
    If $0 < \gamma < 1$, does the system has a unique solution?
\end{openproblem}

In the initial biological formulation presented in \cite{degaetano2024}, the analysis was 
primarily oriented towards parameter identification, physiological simulation, and local 
stability. A comprehensive mathematical treatment of the global dynamic behavior for both systems 
\eqref{eq:non_auto_1}-\eqref{eq:non_auto_6}, \eqref{eq:initial_cond}, 
and \eqref{eq:auto_1}-\eqref{eq:auto_5}, \eqref{eq:initial_cond}  
remains an important open question.

In this paper, we extend and complement the qualitative analysis of the model in two main 
directions:

First, for the  system \eqref{eq:non_auto_1}-\eqref{eq:non_auto_6} 
where biological age 
evolves as $A(t) = t + A_0$, we establish the following main outcomes: 
\begin{itemize}
    \item \textbf{Global existence, uniqueness, and positivity:} We prove that for any $\gamma \ge 1$ and all 
    positive initial values, the system possesses a unique $C^1$ solution defined for all $t \ge 0$, and all 
    state variables $G(t), B(t), I(t), S(t), L(t)$ remain strictly positive for all $t \ge 0$.
    \item \textbf{Asymptotic decay of sensitivity parameters:} Due to the explicit time-dependent biological 
    age term, the peripheral and hepatic insulin sensitivities decay to zero as $t \to \infty$, 
    with asymptotic behavior $S(t) = O(1/t)$ and $L(t) = O(1/t)$. 
    \item \textbf{Convergence of plasma glucose:} Despite the progressive decay of insulin sensitivity, the 
    plasma glucose concentration $G(t)$ remains bounded for all $t \ge 0$ and converges globally to the constant 
    physiological threshold $G_u + \frac{k_1}{s}$ as $t \to \infty$. 
\end{itemize}

Second, regarding the autonomous system \eqref{eq:auto_1}-\eqref{eq:auto_5}  
with fixed age parameter $A_0$, we 
prove the following results: 
\begin{itemize}
    \item \textbf{Global existence, uniqueness, and positivity:} We establish that for any $\gamma \ge 1$ and 
    all positive initial conditions, the system admits a unique $C^1$ solution defined for all $t \ge 0$, with 
    all state variables $G(t), B(t), I(t), S(t), L(t)$ remaining strictly positive.  
    \item \textbf{Boundedness of insulin sensitivities:} In contrast to the aging model, the peripheral and 
    hepatic insulin sensitivities $S(t)$ and $L(t)$ remain strictly bounded away from zero and infinity for all 
    $t \ge 0$.  
    \item \textbf{Global upper bounds under positive lower bound on glucose:} Under the physiological assumption 
    that glucose concentration is bounded below by a positive constant $c_0 > 0$ 
    (see Open Problem \ref{open:G bounded below}), 
    we prove that the pancreatic 
    $\beta$-cell mass $B(t)$ and plasma insulin concentration $I(t)$ remain uniformly bounded for all $t \ge 0$. 
    \item \textbf{Existence and uniqueness of positive equilibrium:} For constant input functions $Y(t) = Y_0$ 
    and $F(t) = F_0$, we prove that the system possesses a unique positive equilibrium point $(G_*, B_*, I_*, 
    S_*, L_*)$.  
    \item \textbf{Local asymptotic stability:} 
    Using a rigorous spectral analysis and the Routh–Hurwitz 
    stability criterion, we establish that this unique positive equilibrium point 
    is locally asymptotically 
    stable for all positive system parameters.  
\end{itemize}

The remainder of this paper is organized as follows. In \autoref{sec:Exist_and_unique}, 
we establish the 
global existence and uniqueness of $C^1$ solutions for both systems. \autoref{sec:Positivity} is devoted 
to proving the 
positivity of state variables. 
In \autoref{sec:qualitative}, we examine the main qualitative properties, asymptotic behavior, 
and local stability of the systems.

%%%%%%%%%%%%%%%%%%%%%%%%%%%%%%%%%%%%%%%%%%%%%%%%%%%%%%%%%%%%
\section{Existence and uniqueness}\label{sec:Exist_and_unique}
%%%%%%%%%%%%%%%%%%%%%%%%%%%%%%%%%%%%%%%%%%%%%%%%%%%%%%%%%%%%
\begin{lemma}[\cite{Agarwal2001FixedPoint}, Theorem 1.7]\label{lem:Agarwal}
Let $I=[0,b]$, $f:I\times\mathbf{R}^n\to\mathbf{R}^n$ is continuous and 
Lipschitz in the second variable, that is, there exists $\alpha\geq 0$ such that 
\begin{equation}
\Vert f(t,y)-f(t,z)\Vert\leq \alpha \Vert y-z\Vert\quad\text{for all }y,\,z
\in\mathbf{R}^n. 
\end{equation}
Then there exists a unique $y\in C^1(I)$ that solves the initial value problem
\begin{align*}
y'(t) &= f(t,y(t)),\\
y(0) &=y_0.
\end{align*}
\end{lemma}
We now prove that our two initial value problems have unique solutions.
\begin{lemma}
    The system \eqref{eq:non_auto_1}-\eqref{eq:non_auto_6}, \eqref{eq:initial_cond}  
     has a unique $C^1$ solution for $\gamma\geq 1$, and this solution is defined for all $t\geq 0$. 
\end{lemma}
\begin{proof}
Let $I=[0,b]$ where $b>0$ is arbitrary. In this case 
\begin{equation}
    f(t,y)=[f_1(t,y),f_2(t,y),f_3(t,y),f_4(t,y),f_5(t,y)]^T,
\end{equation}
where 
\begin{equation}
    y=[y_1,y_2,y_3,y_4,y_5]^T,
\end{equation}
\begin{align}
    f_1(t,y) &=k_1 e^{-\alpha_0 y_5 y_3 y_1} - k_2 y_4 y_3 y_1 - u(y_1), \label{eq:f_1} \\
    f_2(t,y) &= k_3 e^{-\lambda t} - k_4 y_1 y_2, \label{eq:f_2} \\
    f_3(t,y) &= -k_5 y_3 + k_6 \frac{y_2}{B_0} \frac{\vert{}y_1\vert{}^\gamma}{c_1 + \vert{}y_1\vert{}^\gamma}, \label{eq:f_3} \\
    f_4(t,y) &= k_7 - k_8 y_4 + k_9 Y(t) - k_{10} F(t) y_4 - k_{11} t y_4, \label{eq:f_4} \\
    f_5(t,y) &= k_{12} - k_{13} y_5 + k_{14} Y(t) - k_{15} F(t) y_5 - k_{16} t y_5, \label{eq:f_5}
\end{align}
and
\begin{equation*}
    y_0=[y_1(0),y_2(0),y_3(0),y_4(0),y_5(0)]^T=[G_0,B_0,I_0,S_0,L_0]^T.
\end{equation*}
In a finite-dimensional space, any two norms are equivalent; therefore, it suffices to show that 
$f_1(t,y),\ldots,f_5(t,y)$ are Lipschitz continuous with respect to the variable $y$. It is enough 
to prove that $u(y_1)$ and 
\begin{equation}
    h(y_1)=\frac{|y_1|^{\gamma}}{c_1+|y_1|^{\gamma}}, (\gamma \geq 1)
\end{equation}
are Lipschitz continuous, because the other functions are continuously differentiable. 
Obviously $|u(y_1)-u(z_1)|\leq s|y_1-z_1|$. Now consider the function 
$h(y_1)$. Since $0\leq t\leq b$ we have to prove that $h$ is Lipschitz continuous on $[-a,a]$ where 
$a$ is arbitrary fixed. If $\gamma>1$ then $h$ is continuously differentiable as we observed in 
 \autoref{rem:frac_gamma}, so $h$ is Lipschitz continuous. Let $\gamma=1$. We show that the function $h$ 
 is Lipschitz continuous. First, we show that $h$ is continuously differentiable on $[-a,0]$ and on 
 $[0,a]$. Since $h$ is an even function we consider only the interval $[0,a]$. In this case
 \begin{equation}
    h(y_1)=\frac{y_1}{c_1+y_1}.
\end{equation}
If $y_1\in (0,a]$ then 
\begin{equation}
    h'(y_1)=\frac{c_1}{(c_1+y_1)^2}.
\end{equation}
The right-hand-side derivative at $0$ is
\begin{equation}
    h'(0+)=\lim_{y_1 \to 0+} \frac{h(y_1)-h(0)}{y_1-0}=\frac{1}{c_1}.
\end{equation}
On the other hand
\begin{equation}
    \lim_{y_1 \to 0+} h'(y_1)=\frac{1}{c_1}=h'(0+).
\end{equation}
Thus $h$ is continuously differentiable on $[-a,0]$ and on 
 $[0,a]$ which implies the Lipschitz continuity on these two intervals. Obviously if 
 $y\geq 0$ then $|h'(y)|\leq 1/c_1$. By the Lagrange Mean Value Theorem, the function $h$ 
 is Lipschitz continuous on $[-a,0]$ and on $[0,a]$ with Lipschitz constant $L:=1/c_1$. Now let 
$-a\leq y_1<0$ and $a\geq y_2>0$. Since $h$ is an even function $h(y_1)=h(-y_1)$. Here $-y_1\in (0,a]$. 
Then applying the above proved $L$-Lipschitz estimation we obtain
\begin{align*}
    |h(y_1)-h(y_2)| &=|h(-y_1)-h(y_2)| \\
                    &\leq |h(-y_1)-h(0)|+|h(0)-h(y_2)| \\
                    &\leq L (-y_1)+ L y_2 \\
                    &= L(y_2-y_1)=L|y_2-y_1| \\
                    &= L|y_1-y_2|.
\end{align*}
Thus we verified that $h$ is Lipschitz continuous for $\gamma=1$. Now let’s assume that $0<\gamma<1$. 
We show that $h$ is not Lipschitz continuous for any $[0,a]$ interval. We prove indirectly. 
Assume that there exist $a>0$ and $L>0$ such that for any $y_1,y_2\in [0,a]$ we have
\begin{equation}\label{eq:Lip_ind}
    |h(y_1)-h(y_2)|\leq L |y_1-y_2|.
\end{equation}
Let’s choose $y_2=0$. Since $h(0)=0$, formula \eqref{eq:Lip_ind} is reduced to 
\begin{equation*}
    \vert{}h(y_1) - 0\vert{} \le L \vert{}y_1 - 0\vert{} \implies h(y_1) \le L y_1,
\end{equation*}
that is,
\begin{equation*}
    \frac{y_1^{\gamma - 1}}{c_1 + y_1^\gamma} \leq L.
\end{equation*}
Since
\begin{equation*}
    \lim_{y_1\to 0+} \frac{y_1^{\gamma - 1}}{c_1 + y_1^\gamma}=+\infty
\end{equation*}
which is impossible. We obtained contradiction, so we proved that $h$ is 
not Lipschitz continuous for $0<\gamma<1$. 
 
Since $0\leq t\leq b$ and $Y(t)$, $F(t)$ are bounded functions, 
from the previous results it follows that the function $f(t,y)$ 
is Lipschitz continuous in the second variable for $\gamma\geq 1$ independently of  $t\in [0,b]$, 
but the Lipschitz constant may depend on $b$. 
Since $b>0$ is arbitrary and the solution is unique we obtain that the solution is 
defined for all $t\geq 0$. 
\end{proof}

In the same way we obtain 
\begin{lemma}
    The system \eqref{eq:auto_1}-\eqref{eq:auto_5}, \eqref{eq:initial_cond} 
    has a unique $C^1$ solution for $\gamma\geq 1$, and this solution is defined for all $t\geq 0$.
\end{lemma}

%%%%%%%%%%%%%%%%%%%%%%%%%%%%%%%%%%%%%%%%%%%%%%%%%%%%%%%%%%%%
\section{Positivity of the solution}\label{sec:Positivity}
%%%%%%%%%%%%%%%%%%%%%%%%%%%%%%%%%%%%%%%%%%%%%%%%%%%%%%%%%%%%

First, we will prove that $G(t) > 0$ ($t\geq 0$)  for both systems. 
Then we will discuss the two systems separately. 
\begin{lemma}\label{lem:u_prime}
Let $u:[0,\infty)\to\mathbb{R}$ be a $C^1$ function. Let's assume that $u(0) > 0$. 
Suppose that there exist $0<x_1<x_2$ such that $u(x_1)=0$ and $u(x)<0$ for all $x\in(x_1,x_2]$. 
Then $u'(x_1)\leq 0$.
\begin{proof}
 
Using the definition 
\begin{equation*}
    u'(x_1)=u'(x_1+0) = \lim_{x \to x_1+0} \frac{u(x) - u(x_1)}{x - x_1}.
\end{equation*}
Here $u(x) - u(x_1)=u(x)<0$ and $x-x_1>0$, therefore $u'(x_1)\leq 0$.  
\end{proof}
\end{lemma}
\begin{lemma}\label{lem:G pos}
    The function $G$ is positive for both systems.
\end{lemma}
\begin{proof}
First we prove $G(x)\geq 0$ for $x\geq 0$. 
We prove this indirectly. Let's assume that there exists $x_2$ such that $G(x_2)<0$. Denote
\begin{equation*}
    x_1:=\sup \{x\,|\,0\leq x<x_2, G(x)>0\}.
\end{equation*}
Since $G(0)>0$ it follows $0<x_1$. By the continuity of $G$ we have $G(x_1)=0$ and $x_1<x_2$, and 
$G(x)<0$ for $x_1<x<x_2$. By \autoref{lem:u_prime} we obtain $G'(x_1)\leq 0$. 
However formula \eqref{eq:non_auto_1} gives $G'(x_1)=k_1>0$. This contradiction 
implies $G\geq 0$. Now we prove that $G$ is positive. We prove this indirectly. 
Let's assume that there exists $0<x_0$ such that $G(x_0)=0$. Since $G$ is a $C^1$ function 
we get $G'(x_0)=0$. But formula \eqref{eq:non_auto_1} gives $G'(x_0)=k_1>0$. 
This contradiction implies $G> 0$. 
\end{proof}

\subsection{\texorpdfstring{The system $A(t)=t+A_0$}{The system A(t)=t+A0}}

\begin{lemma}
    $L(t), S(t), I(t), B(t)>0$ for all $t\geq 0$.
\end{lemma}
\begin{proof}
    In this case 
\begin{equation}
L'(t)=k_{12}+k_{14}Y(t)-(k_{15}F(t)+k_{16}t+k_{13})L(t).
\end{equation}
The solution is
\begin{align}
L(t) = & \left(\int_0^t \left( k_{12} + k_{14}Y(z) \right) 
e^{\int_0^z k_{15}F(x)+k_{16}x+k_{13} \, dx} \, dz +L_0 \right) \nonumber \\
   &\times e^{-\int_0^t k_{15}F(x)+k_{16}x+k_{13} \, dx}. \label{eq:L gen exp sol}
\end{align}
Here $k_{12}+k_{14}Y(z)>0$ and $L_0>0$ so 
\begin{equation}
    L(t)>0\quad (t\geq 0).
\end{equation}
Similarly, 
\begin{equation}
S'(t)=k_7+k_9 Y(t)-(k_{10}F(t)+k_{11}t+k_8)S(t)
\end{equation}
The solution is
\begin{align}
S(t) = & \left(\int_0^t \left( k_{7} + k_{9}Y(z) \right) 
e^{\int_0^z k_{10}F(x)+k_{11}x+k_{8} \, dx} \, dz +S_0 \right) \nonumber \\
   &\times e^{-\int_0^t k_{10}F(x)+k_{11}x+k_{8} \, dx}. \label{eq:S gen exp sol}
\end{align}
Here $k_{7}+k_{9}Y(z)>0$ and $S_0>0$ so 
\begin{equation}
    S(t)>0\quad (t\geq 0).
\end{equation}
Now we consider $I(t)$.
Then we obtain
\begin{equation}\label{eq:I gen exp sol}
I(t)=\left(
\int_0^t k_6\frac{B(z)}{B_0}
\frac{(G(z))^{\gamma}}{c_1+(G(z))^{\gamma}}e^{k_5z}\,dz+
I_0
\right)e^{-k_5t}
\end{equation}
Here $B_0,I_0>0$ so
\begin{equation}
   I(t)>0\quad (t\geq 0).
\end{equation}
Now we consider $B(t)$.
Then we obtain
\begin{equation}\label{eq:B gen exp sol}
B(t)=\left(
\int_0^t k_3 e^{-\lambda z +k_4 \int_0^z G(x)\,dx}
\,dz+
B_0
\right)e^{-k_4 \int_0^t G(z)\,dz}
\end{equation}
Here $B_0>0$ so  
\begin{equation}
   B(t)>0\quad (t\geq 0).
\end{equation}
\end{proof}

\subsection{\texorpdfstring{The system $A(t)=A_0$}{The system A(t)=A0}}

\begin{lemma}
    $L(t), S(t), I(t), B(t)>0$ for all $t\geq 0$.
\end{lemma}
\begin{proof}
    In this case 
\begin{equation}
L'(t)=k_{12}+k_{14}Y(t)-(k_{15}F(t)+k_{13})L(t).
\end{equation}
The solution is
\begin{align}
L(t) = & \left(\int_0^t \left( k_{12} + k_{14}Y(z) \right) 
e^{\int_0^z k_{15}F(x)+k_{13} \, dx} \, dz +L_0 \right) \nonumber \\
   &\times e^{-\int_0^t k_{15}F(x)+k_{13} \, dx}. \label{eq:L gen exp sol A0}
\end{align}
Here $k_{12}+k_{14}Y(z)>0$ and $L_0>0$ so 
\begin{equation}
    L(t)>0\quad (t\geq 0).
\end{equation}
Similarly, 
\begin{equation}
S'(t)=k_7+k_9 Y(t)-(k_{10}F(t)+k_8)S(t)
\end{equation}
The solution is
\begin{align}
S(t) = & \left(\int_0^t \left( k_{7} + k_{9}Y(z) \right) 
e^{\int_0^z k_{10}F(x)+k_{8} \, dx} \, dz +S_0 \right) \nonumber \\
   &\times e^{-\int_0^t k_{10}F(x)+k_{8} \, dx}. \label{eq:S gen exp sol A0}
\end{align}
Here $k_{7}+k_{9}Y(z)>0$ and $S_0>0$ so 
\begin{equation}
    S(t)>0\quad (t\geq 0).
\end{equation}
Since equation \eqref{eq:non_auto_3} is the same as \eqref{eq:auto_3} we obtain
\begin{equation}
   I(t)>0\quad (t\geq 0).
\end{equation}
Now we consider $B(t)$. 
Then we obtain
\begin{equation}\label{eq:B gen exp sol A0}
B(t)=\left(
\int_0^t k_3 e^{k_4 \int_0^z G(x)\,dx}
\,dz+
B_0
\right)e^{-k_4 \int_0^t G(z)\,dz}
\end{equation}
Here $B_0>0$ so
\begin{equation}
   B(t)>0\quad (t\geq 0).
\end{equation}
\end{proof}

\section{Qualitative properties of the solutions}\label{sec:qualitative}

First, we set an upper bound on $G(t)$ that applies to both systems.
\begin{theorem}\label{thm:G_upper_est}
    We have for all $t\geq 0$  
\begin{equation*}
    G(t)\leq \max\left(\frac{k_1}{s}+G_u,G_0 \right).
\end{equation*}
\end{theorem}
\begin{proof}
Since $L(t), S(t), I(t)>0$ we get 
\begin{equation*}
    G'(t)\leq k_1-u(G(t)).
\end{equation*}
Now we distinguish four cases.

(a) $G(t)\leq G_u$ for all $t\geq 0$. In this case, the upper bound holds. 

(b) $G(t)\geq G_u$ for all $t\geq 0$. Then we get 
\begin{equation*}
G'(t)\leq (k_1+sG_u)-sG(t)\quad (t\geq 0).
\end{equation*}
By \autoref{cor:diffineq_cor} we have 
\begin{equation*}
G(t)\leq \max \left(\frac{k_1}{s}+G_u,G_0\right) \quad (t\geq 0).
\end{equation*}

(c) There exists $T_0>0$ such that $G(t)\geq G_u$ for all $t\in [0,T_0]$. Then 
by  \autoref{cor:diffineq_cor} we have 
\begin{equation*}
G(t)\leq \max \left(\frac{k_1}{s}+G_u,G_0\right) \quad (t\in [0,T_0]).
\end{equation*}

(d) $G_0\leq G_u$ and there exists $t_0$ such that $G(t_0)>G_u$. Then there exist 
$0\leq\alpha, \beta> t_0$ such that $G(\alpha)=G_u$ and $G(t)>G_u$ for all 
$\alpha <t <\beta$. 
Then we get 
\begin{equation}
G'(t)\leq (k_1+sG_u)-sG(t)\quad (\alpha\leq t<\beta).
\end{equation}
Then \autoref{cor:diffineq_cor} yields 
\begin{align}
G(t) &\leq G_ue^{-s(t-\alpha)}+\frac{k_1+sG_u}{s}
\left(1-e^{-s(t-\alpha)}\right)\\
&=\frac{k_1}{s}+G_u-\frac{k_1}{s}e^{-s(t-\alpha)}\quad (\alpha\leq t<\beta) \nonumber\\
&<G_u+\frac{k_1}{s} \nonumber.
\end{align}
\end{proof}

\subsection{\texorpdfstring{The system $A(t)=t+A_0$}{The system A(t)=t+A0}}

\begin{lemma}\label{lem:lim L S}
\begin{equation*}
     \lim_{t\to\infty}L(t)=0,\quad \lim_{t\to\infty}S(t)=0. 
\end{equation*}
\end{lemma}
\begin{proof}
Obviously
\begin{equation*}
\lim_{t\to\infty} L_0 e^{-\int_0^t k_{15}F(x)+k_{16}x+k_{13} \, dx}=0.
\end{equation*}
\end{proof}
The other term is
\begin{equation*}
\lim_{t\to\infty}
\frac{\int_0^t \left( k_{12} + k_{14}Y(z) \right) 
e^{\int_0^z k_{15}F(x)+k_{16}x+k_{13} \, dx} \, dz }
{e^{\int_0^t k_{15}F(x)+k_{16}x+k_{13} \, dx}}. 
\end{equation*}
Applying the L’Hospital rule we obtain that the limit is
\begin{equation*}
\lim_{t\to\infty}
\frac{\left( k_{12} + k_{14}Y(t) \right) 
e^{\int_0^t k_{15}F(x)+k_{16}x+k_{13} \, dx}}
{e^{\int_0^t k_{15}F(x)+k_{16}x+k_{13} \, dx} (k_{15}F(t)+k_{16}t+k_{13})} =0.
\end{equation*}
In the same way we obtain
\begin{equation*}
\lim_{t\to\infty} S(t)=0.
\end{equation*}

Now we introduce some notations.
\begin{notation}
\begin{equation}
\mathrm{erfi}(x)=\frac{2}{\sqrt{\pi}}\int_0^x e^{y^2}dy
\quad x\geq 0.
\end{equation}
\end{notation}
\begin{notation}
\begin{equation}
\mathrm{daw}(x)=\frac{\sqrt{\pi}}{2}e^{-x^2}\mathrm{erfi}(x)
\quad x\geq 0.
\end{equation}
\end{notation}
\begin{lemma}[\cite{AtlasOfFunctions}, 42:6:3] For all $x>0$ 
\begin{equation}
\mathrm{daw}(x)=\frac{1}{2x}\sum_{j=0}^{n-1}\frac{(2j-1)!!}{(2x^2)^j}
+\frac{(2n-1)!!x}{(2x^2)^n \exp(x^2)}\sum_{j=0}^{\infty} 
\frac{x^{2j}}{(2j-2n+1)j!}\quad n=1,2,\ldots.
\end{equation}
\end{lemma}

\begin{lemma}\label{lem:daw_x_estimation} For all $x>0$ 
\begin{equation}
\left| \mathrm{daw}(x)- \frac{1}{2x}-\frac{1}{4x^3}\right| \leq 
\frac{9}{4x^5}.
\end{equation}
\end{lemma}
\begin{proof}
Let $n=3$. Then
\begin{equation*}
\mathrm{daw}(x)=\frac{1}{2x}+\frac{1}{4x^3}+\frac{3}{8x^5}+
\frac{15}{8x^5 \exp(x^2)}\sum_{j=0}^{\infty}\frac{x^{2j}}{(2j-5)j!}.
\end{equation*}
Since
\begin{equation*}
\frac{1}{\exp(x^2)}\sum_{j=0}^{\infty}\frac{x^{2j}}{j!|2j-5|}\leq 1
\end{equation*}
we obtain the statement.
\end{proof}

\begin{lemma}[\cite{Emmrich2004}, Lemma 7.3.2]\label{lem:Gronwall}
Assume that $T\in(0,\infty)$, $y:[0,T)\to\mathbf{R}$ is absolutely 
continuous and $a,\,b\in L^1(0,T)$. If the differential inequality 
\begin{equation}
y'(t)\leq a(t)y(t)+b(t),\quad t\in (0,T)
\end{equation}
is satisfied then the estimate
\begin{equation}
y(t)\leq y(0)e^{A(t)}+\int_0^t e^{A(t)-A(s)}b(s)\,ds, \quad t\in[0,T)
\end{equation}
holds, where $A(t)=\int_0^t a(x)\,dx$.
\end{lemma}

\begin{corollary}\label{cor:diffineq_cor}
If $a>0$ and 
\begin{equation*}
y'(t)\leq -ay(t)+b, \quad t\in[0,T);
\end{equation*}
then
\begin{align*}
y(t) &\leq y(0)e^{-at}+\frac{b}{a}\left(1-e^{-at} \right) \\
&= \frac{b}{a}+\left(y(0)-\frac{b}{a} \right)e^{-at}, \quad t\in[0,T).
\end{align*}
That is, if $y(0)\leq b/a$, then $y(t)\leq b/a$, while if $y(0)>b/a$ 
then $y(t)\leq y(0)$ for all $t\geq 0$.
\end{corollary}

\begin{corollary}\label{cor:upper_diff_ineq}
If $a>0$ and 
\begin{equation}
y'(t)\leq -(at+c)y(t)+b,\quad t\in[0,T);
\end{equation}
then for large $t$ we have  
\begin{equation}
y(t)\leq \frac{b}{a} \cdot \frac{1}{t} - \frac{bc}{a^2} \cdot \frac{1}{t^2} + \frac{b(c^2 + a)}{a^3} \cdot \frac{1}{t^3} - \frac{bc(c^2 + 3a)}{a^4} \cdot \frac{1}{t^4} + O\left(\frac{1}{t^5}\right).
\end{equation}
\end{corollary}
\begin{proof}
By \autoref{lem:Gronwall} we have 
\begin{align*}
y(t) &\leq y(0)e^{-at^2/2-ct}+be^{-at^2/2-ct}
\int_0^t e^{as^2/2+cs}\,ds \\
&=y(0)e^{-at^2/2-ct}+b\,e^{-((at+c)/\sqrt{2a})^2}
\int_0^t e^{((as+c)/\sqrt{2a})^2}\, ds\\
&=y(0)e^{-at^2/2-ct}+b\sqrt{\frac{2}{a}}\,e^{-((at+c)/\sqrt{2a})^2}
\int_{\frac{c}{\sqrt{2a}}}^{\frac{at+c}{\sqrt{2a}}} e^{z^2}\,dz\\
&= y(0)e^{-at^2/2-ct}-b\sqrt{\frac{2}{a}}\,e^{-((at+c)/\sqrt{2a})^2}
\int_0^{\frac{c}{\sqrt{2a}}} e^{z^2}\,dz\\
&\quad +b\sqrt{\frac{2}{a}}\,e^{-((at+c)/\sqrt{2a})^2}
\int_0^{\frac{at+c}{\sqrt{2a}}} e^{z^2}\,dz\\
&=y(0)e^{-at^2/2-ct}-b\sqrt{\frac{2}{a}}\,e^{-((at+c)/\sqrt{2a})^2}
\int_0^{\frac{c}{\sqrt{2a}}} e^{z^2}\,dz\\
&\quad +b\sqrt{\frac{2}{a}} \mathrm{daw}\left( 
\frac{at+c}{\sqrt{2a}} \right).
\end{align*}
Applying \autoref{lem:daw_x_estimation} we obtain
\begin{align*}
y(t) &\leq b\sqrt{\frac{2}{a}} 
\left(
\frac{\sqrt{2a}}{2(at+c)}+\frac{(2a)^{3/2}}{4(at+c)^3}
\right) + O\left(\frac{1}{t^5} \right) \\
&=\frac{b}{a} \cdot \frac{1}{t} - \frac{bc}{a^2} \cdot \frac{1}{t^2} + \frac{b(c^2 + a)}{a^3} \cdot \frac{1}{t^3} - \frac{bc(c^2 + 3a)}{a^4} \cdot \frac{1}{t^4} + O\left(\frac{1}{t^5}\right).
\end{align*}
\end{proof}

\begin{lemma}\label{lem:reverse Gronwall}
Assume that $T\in(0,\infty)$, $y:[0,T)\to\mathbf{R}$ is absolutely 
continuous and $a,\,b\in L^1(0,T)$. If the differential inequality 
\begin{equation}
y'(t)\geq a(t)y(t)+b(t),\quad t\in (0,T)
\end{equation}
is satisfied then the estimate
\begin{equation}
y(t)\geq y(0)e^{A(t)}+\int_0^t e^{A(t)-A(s)}b(s)\,ds, \quad t\in[0,T)
\end{equation}
holds, where $A(t)=\int_0^t a(x)\,dx$.
\end{lemma}
\begin{proof}
If 
\begin{equation}
y'(t)\geq a(t)y(t)+b(t),\quad t\in (0,T)
\end{equation}
then
\begin{equation}
(-y(t))'\leq a(t)(-y(t))+(-b(t)),\quad t\in (0,T).
\end{equation}
Applying \autoref{lem:Gronwall} we obtain the statement.
\end{proof}

\begin{corollary}\label{cor:reversediffineq_cor}
If $a>0$ and 
\begin{equation*}
y'(t)\geq -ay(t)+b, \quad t\in[0,T);
\end{equation*}
then
\begin{align*}
y(t) &\geq y(0)e^{-at}+\frac{b}{a}\left(1-e^{-at} \right) \\
&= \frac{b}{a}+\left(y(0)-\frac{b}{a} \right)e^{-at}, \quad t\in[0,T).
\end{align*}
That is, if $y(0)\leq b/a$, then $y(t)\geq y(0)$, while if $y(0)>b/a$ 
then $y(t)\geq b/a$ for all $t\in [0,T)$.
\end{corollary}

\begin{corollary}\label{cor:lower_diff_ineq}
If $a>0$ and 
\begin{equation}
y'(t)\geq -(at+c)y(t)+b,\quad t\in[0,T);
\end{equation}
then for large $t\in[0,T)$ we have  
\begin{equation}
y(t)\geq \frac{b}{a} \cdot \frac{1}{t} - \frac{bc}{a^2} \cdot \frac{1}{t^2} + \frac{b(c^2 + a)}
{a^3} \cdot \frac{1}{t^3} - \frac{bc(c^2 + 3a)}{a^4} \cdot \frac{1}{t^4} + O\left(\frac{1}
{t^5}\right).
\end{equation}
\end{corollary}

\begin{theorem}\label{thm:B_t_upper_estimation}
For $t\geq 0$ we have
\begin{equation*}
B(t)\leq B_0+\frac{k_3}{\lambda}
\left(1 -e^{-\lambda t}\right).
\end{equation*} 
\end{theorem}

\begin{proof}
We use \eqref{eq:B gen exp sol}. 
Since 
\begin{align*}
&\left(
\int_0^t k_3 e^{-\lambda z +k_4 \int_0^z G(x)\,dx}
\,dz
\right)e^{-k_4 \int_0^t G(z)\,dz} \\
&\leq \int_0^t k_3 e^{-\lambda z}=\frac{k_3}{\lambda}
\left(1 -e^{-\lambda t}\right), 
\end{align*}
the upper estimation follows.
\end{proof}

\begin{theorem}
For $t\geq 0$ we have
\begin{equation}
I(t)\leq  \frac{k_6}{k_5}\left(1+\frac{k_3}{\lambda B_0} \right)
\frac{G_{\max}^{\gamma}}{c_1+G_{\max}^{\gamma}}
+I_0e^{-k_5 t},
\end{equation} 
where
\begin{equation*}
G_{\max}=\max\left(\frac{k_1}{s}+G_u,G_0 \right),
\end{equation*}
\end{theorem}

\begin{proof}
We use \eqref{eq:I gen exp sol} and \autoref{thm:B_t_upper_estimation}. Since 
\begin{equation*}
k_6\frac{B(z)}{B_0}
\frac{(G(z))^{\gamma}}{c_1+(G(z))^{\gamma}}\leq 
k_6\left(1+\frac{k_3}{\lambda B_0} \right)
\frac{G_{\max}^{\gamma}}{c_1+G_{\max}^{\gamma}},
\end{equation*}
where
\begin{equation*}
G_{\max}=\max\left(\frac{k_1}{s}+G_u,G_0 \right),
\end{equation*}
the upper estimation follows.
\end{proof}

\begin{theorem}\label{thm:L_t_estimation}
For large $t$ we have 
\begin{equation}
\frac{k_{12}+k_{14}c_2}{k_{16}}\cdot\frac{1}{t}+O\left(\frac{1}{t^2} \right)\leq 
L(t)\leq \frac{k_{12}+k_{14}c_3}{k_{16}}\cdot\frac{1}{t}+O\left(\frac{1}{t^2} \right)
\end{equation}
\end{theorem}

\begin{proof}
Since $c_2\leq Y(t)\leq c_3$ and $c_4\leq F(t)\leq c_5$ we can write 
\begin{align*}
L'(t) &\leq k_{12}-k_{13}L(t)+k_{14}c_3-k_{15}c_4 L(t) -k_{16}t L(t) \\
&= -(k_{16}t+k_{13}+k_{15}c_4)L(t)+(k_{12}+k_{14}c_3),
\end{align*}
and
\begin{align*}
L'(t) &\geq k_{12}-k_{13}L(t)+k_{14}c_2-k_{15}c_5 L(t)-k_{16}t L(t) \\
&= -(k_{16}t+k_{13}+k_{15}c_5)L(t)+(k_{12}+k_{14}c_2).
\end{align*}
Denote
\begin{align*}
a &:=k_{16}, \\
c &:=k_{13}+k_{15}c_4, \\
b &:=k_{12}+k_{14}c_3
\end{align*}
Then by \autoref{cor:upper_diff_ineq} we get 
\begin{equation}
L(t)\leq \frac{k_{12}+k_{14}c_3}{k_{16}}\cdot\frac{1}{t}+O\left(\frac{1}{t^2} \right).
\end{equation}
Similarly, denote
\begin{align*}
a &:=k_{16}, \\
c &:=k_{13}+k_{15}c_5, \\
b &:=k_{12}+k_{14}c_2.
\end{align*}
Then by \autoref{cor:lower_diff_ineq} we get
\begin{equation}
L(t)\geq \frac{k_{12}+k_{14}c_2}{k_{16}}\cdot\frac{1}{t}+O\left(\frac{1}{t^2} \right).
\end{equation}
\end{proof}

\begin{lemma}\label{lem:upper_diff_ineq_gen}
If $a>0$ $t_0\geq 0$, and 
\begin{equation*}
y'(t)\leq -(at+c)y(t)+(bt+d),  t\in[t_0,T);
\end{equation*}
then for large $t$ we have
\begin{equation*}
y(t)\leq C,
\end{equation*}
where $C$ is a constant. 
\end{lemma}

\begin{proof}
We give a direct proof. Multiplying by $e^{(a/2)t^2+ct}$ the inequality 
\begin{equation*}
y'(t)+(at+c)y(t)\leq bt+d
\end{equation*}
we obtain
\begin{equation*}
\left(y(t)e^{(a/2)t^2+ct}  \right)'\leq (bt+d)e^{(a/2)t^2+ct}.
\end{equation*}
After integrating both sides on the interval $[t_0,t]$ and rearranging we get
\begin{equation*}
y(t) \le y(t_0) e^{-(a/2)(t^2 - t_0^2) - c(t - t_0)} + e^{-(a/2)t^2 - ct} \int_{t_0}^t (bs + d) 
e^{(a/2)s^2 + cs} \, ds.
\end{equation*}
Here 
\begin{equation*}
\int_{t_0}^t (bs + d) e^{(a/2)s^2 + cs} \, ds = \frac{b}{a} \int_{t_0}^t (as + c) e^{(a/2)s^2 + 
cs} \, ds + \left(d - \frac{bc}{a}\right) \int_{t_0}^t e^{(a/2)s^2 + cs} \, ds. 
\end{equation*}
The first integral is
\begin{equation*}
\int_{t_0}^t (as + c) e^{(a/2)s^2 + cs} \, ds = e^{(a/2)t^2 + ct} - e^{(a/2)t_0^2 + ct_0}
\end{equation*}
Using direct calculation (integration by parts) 
or \autoref{lem:daw_x_estimation}, for large $t$ values we have 
\begin{equation*}
\int_{t_0}^t e^{(a/2)s^2 + cs} \, ds \le K + \frac{e^{\frac{a}{2}t^2 + ct}}{at+c}.
\end{equation*}
Using these estimation we obtain 
\begin{equation*}
    y(t)\leq C.
\end{equation*}

In a similar way we can derive
\begin{lemma}\label{lem:lower_diff_ineq_gen}
If $a>0$ $t_0\geq 0$, and 
\begin{equation*}
y'(t)\geq -(at+c)y(t)+(bt+d),  t\in[t_0,T);
\end{equation*}
then for large $t$ we have
\begin{equation*}
y(t)\geq C,
\end{equation*}
where $C$ is a constant. 
\end{lemma}
\end{proof}

\begin{theorem}\label{thm:L asymp}
In addition to the previous assumptions, given that $Y'(t), F'(t)$ are bounded, 
we have 
\begin{equation*} 
    L(t)=\frac{k_{12}+k_{14}Y(t)}{k_{16}}\cdot\frac{1}{t}+O\left( \frac{1}{t^2} \right).
\end{equation*}
\end{theorem}
\begin{proof}
Motivating by \autoref{thm:L_t_estimation} we try to find $L(t)$ in the following 
form 
\begin{equation}
L(t)=\frac{k_{12}+k_{14}Y(t)}{k_{16}}\cdot\frac{1}{t}+\frac{\ell(t)}{t^2},
\end{equation}
where 
\begin{equation*}
\ell(t)=O(1).
\end{equation*}
Then
\begin{equation}\label{eq:L_diff_lhs}
L'(t)=\frac{k_{14}Y'(t)}{k_{16}t}-\frac{k_{12}+k_{14}Y(t)-k_{16}\ell'(t)}{k_{16}t^2}
-\frac{2\ell(t)}{t^3},
\end{equation}
and
\begin{align}\label{eq:L_diff_rhs}
k_{12}+k_{14}Y(t)-(k_{15}F(t)+k_{16}t+k_{13})L(t) =
&-\frac{(k_{15}F(t)+k_{13})(k_{12}+k_{14}Y(t))+k_{16}^2\ell(t)}{k_{16}t} 
\nonumber \\
&-\frac{(k_{15}F(t)+k_{13})\ell(t)}{t^2}.
\end{align}
Let
\begin{equation*}
\ell(t):=\ell_1(t)+\ell_2(t),
\end{equation*}
and
\begin{equation}
\ell_1(t):=-\frac{(k_{15}F(t)+k_{13})(k_{12}+k_{14}Y(t))}{k_{16}^2}. 
\end{equation}
Then we obtain
\begin{align}\label{eq:L prime long}
L'(t) &=
\frac{1}{t}\frac{k_{14}Y'(t)}{k_{16}} \nonumber
\\
&+\frac{1}{t^2}\left[ \ell_2'(t)
-\frac{
k_{15}F'(t)(k_{12}+k_{14}Y(t))
+(k_{15}F(t)+k_{13})k_{14}Y'(t)
}{k_{16}^2}
-\frac{k_{12}+k_{14}Y(t)}{k_{16}}
\right] \nonumber
\\
&+\frac{2}{t^3}\left[
\frac{(k_{15}F(t)+k_{13})(k_{12}+k_{14}Y(t))}{k_{16}^2}-\ell_2(t)
\right],
\end{align}
and
\begin{align}\label{eq:L prime rhs}
&k_{12}+k_{14}Y(t)-(k_{15}F(t)+k_{16}t+k_{13})L(t) = \nonumber\\
&\left( -k_{16}\ell_2(t) \right) \frac{1}{t} + \left( \frac{(k_{15}F(t)+k_{13})^2 (k_{12}+k_{14}Y(t))}{k_{16}^2} - (k_{15}F(t)+k_{13})\ell_2(t) \right) \frac{1}{t^2}.
\end{align}

From \eqref{eq:L prime long} and \eqref{eq:L prime rhs} it follows 
\begin{align} 
\ell_2'(t) &= \frac{k_{12}+k_{14}Y(t)}{k_{16}} 
\left[ 1 + \frac{k_{15}F'(t) + (k_{15}F(t)+k_{13})^2}{k_{16}} 
- \frac{2(k_{15}F(t)+k_{13})}{k_{16}t} \right] \nonumber \\ 
& + \frac{k_{14}Y'(t)}{k_{16}} \left[ -t + \frac{k_{15}F(t)+k_{13}}{k_{16}} \right] 
 + \left( -k_{16}t - k_{15}F(t) - k_{13} + \frac{2}{t} \right) \ell_2(t) 
\end{align}
Since $Y(t),F(t),Y'(t),F'(t)$ are bounded we can apply \autoref{lem:upper_diff_ineq_gen} 
and \autoref{lem:lower_diff_ineq_gen} to obtain that $\ell_2(t)=O(1)$. With $\ell_1(t)=O(1)$ 
it follows 
\begin{equation*}
\ell(t)=O(1). 
\end{equation*}
The theorem is proved. 
\end{proof}

After making minor changes we have
\begin{theorem}\label{thm:S_t_estimation}
For large $t$ we have 
\begin{equation*}
\frac{k_{7}+k_{9}c_2}{k_{11}}\cdot\frac{1}{t}+O\left(\frac{1}{t^2} \right)\leq 
S(t)\leq \frac{k_{7}+k_{9}c_3}{k_{11}}\cdot\frac{1}{t}+O\left(\frac{1}{t^2} \right).
\end{equation*}
\end{theorem}

\begin{lemma}
In addition to the previous assumptions, given that $Y'(t), F'(t)$ are bounded, 
we have 
\begin{equation*} 
    S(t)=\frac{k_{7}+k_{9}Y(t)}{k_{11}}\cdot\frac{1}{t}+O\left( \frac{1}{t^2} \right).
\end{equation*}
\end{lemma}

\begin{theorem}\label{thm:lim_G}
We have the following limit 
\begin{equation*}
{\lim_{t\to\infty} G(t)=G_u+\frac{k_1}{s}.}
\end{equation*} 
\end{theorem}
\begin{proof}
First we show that 
\begin{equation}\label{eq:limsup_greater}
\limsup_{t\to\infty} G(t)\geq G_u+\frac{k_1}{s}.
\end{equation}
Indirectly, assume that 
\begin{equation}
\limsup_{t\to\infty} G(t)<G_u+\frac{k_1}{s}.
\end{equation}
Then there exists $0<\varepsilon<k_1/(2s)$ such that 
\begin{equation}
G(t)<G_u+\frac{k_1}{s}-\varepsilon,\qquad t\geq t_0(\varepsilon).
\end{equation}
Then 
\begin{equation}
u(G(t))<k_1-s\varepsilon,\qquad t\geq t_0(\varepsilon).
\end{equation}
Since 
\begin{equation}
\lim_{t\to\infty} k_1e^{-\alpha L(t)I(t)G(t)}-k_2S(t)I(t)G(t)=k_1,
\end{equation}
there exists $t_1(\varepsilon,s)\geq t_0(\varepsilon)$
such that 
\begin{equation}
k_1e^{-\alpha L(t)I(t)G(t)}-k_2S(t)I(t)G(t)>k_1-\frac{s\varepsilon}{2}
, \qquad t\geq t_1(\varepsilon,s).
\end{equation}
Thus for $t\geq t_1(\varepsilon,s)$ we have
\begin{align*}
G'(t) &\geq \frac{s\varepsilon}{2},
\end{align*}
but this is impossible because $G(t)$ is bounded, hence 
the inequality \eqref{eq:limsup_greater} is proved.

Secondly we show that 
\begin{equation}\label{eq:liminf_tau_Gu_strict}
\liminf_{t\to\infty}G(t)> G_u.
\end{equation}
Assume that 
\begin{equation}
\liminf_{t\to\infty} G(t)<G_u.
\end{equation}
Then there exists a number $p$ such that 
\begin{equation}
\liminf_{t\to\infty}G(t)<p<G_u.
\end{equation}
By the definition of $\liminf$ it follows that there exist bounded intervals 
$[a_n,b_n]$ such that $b_n<a_{n+1}$, 
$\displaystyle{\lim_{n\to\infty}}a_n=\infty$, $G(a_n),\,G(b_n)=p$ and 
by Rolle's mean value theorem there 
exists $\tau_n\in (a_n,b_n)$ such that $G'(\tau_n)=0$ and 
$G(\tau_n)\leq p$.  
Substituting $\tau_n$ into the equation of $G'(t)$ in 
\eqref{eq:non_auto_1} we get
\begin{align*}\label{eq:liminf_tau_Gu}
0&=k_1e^{-\alpha L(\tau_n)I(\tau_n)G(\tau_n)}-
k_2S(\tau_n)I(\tau_n)G(\tau_n)-u(G(\tau_n))\\
&=k_1e^{-\alpha L(\tau_n)I(\tau_n)G(\tau_n)}-
k_2S(\tau_n)I(\tau_n)G(\tau_n).
\end{align*}
But this is impossible because the right hand side converges to $k_1$ 
as $n$ tends to infinity. \\
Assume that 
\begin{equation}
\liminf_{t\to\infty} G(t)=G_u.
\end{equation}
Let $0<\delta<k_1/(2s)$ be arbitrary.  
By the definition of $\liminf$ it follows that there exist bounded intervals 
$[a_n,b_n]$ such that $b_n<a_{n+1}$, 
$\displaystyle{\lim_{n\to\infty}}a_n=\infty$, $G(a_n),\,G(b_n)=G_u+
\delta$ and 
by Rolle's mean value theorem there 
exists $\tau_n\in (a_n,b_n)$ such that $G'(\tau_n)=0$ and 
$G(\tau_n)\leq G_u+\delta$.  
Substituting $\tau_n$ into the equation of $G'(t)$ in 
\eqref{eq:non_auto_1} we get
\begin{equation}\label{eq:liminf_tau2}
0=k_1e^{-\alpha L(\tau_n)I(\tau_n)G(\tau_n)}-
k_2S(\tau_n)I(\tau_n)G(\tau_n)-u(G(\tau_n)).
\end{equation}
Here 
\begin{equation*}
u(G(\tau_n))\leq s\delta<\frac{k_1}{2}.
\end{equation*}
Thus if $n\geq n_0$ then the right hand is greater than $k_1/4$ and 
cannot be zero. \\
Hence we proved \eqref{eq:liminf_tau_Gu_strict}.

Thirdly we prove that
\begin{equation}\label{eq:liminf_lower_estimation}
\liminf_{t\to\infty}G(t)\geq G_u+\frac{k_1}{s}.
\end{equation}
Indirectly, assume that 
\begin{equation}
\liminf_{t\to\infty}G(t)< G_u+\frac{k_1}{s}.
\end{equation}
Then there exists a number $p$ such that 
\begin{equation}\label{eq:p_less_than}
G_u< \liminf_{t\to\infty}G(t)<p<G_u+\frac{k_1}{s}.
\end{equation}
By the definition of $\liminf$ it follows that there exist bounded intervals 
$[a_n,b_n]$ such that $b_n<a_{n+1}$, 
$\displaystyle{\lim_{n\to\infty}}a_n=\infty$, $G(a_n),\,G(b_n)=p$ and 
by Rolle's mean value theorem there 
exist $\tau_n\in (a_n,b_n)$ such that $G'(\tau_n)=0$ and 
$G(\tau_n)<p$. If $n\geq n_0$ then $G(\tau_n)>G_u$. 
Substituting $\tau_n$ into the equation of $G'(t)$ in 
\eqref{eq:non_auto_1} we get
\begin{equation}\label{eq:liminf_tau}
0=k_1e^{-\alpha L(\tau_n)I(\tau_n)G(\tau_n)}-
k_2S(\tau_n)I(\tau_n)G(\tau_n)-u(G(\tau_n)).
\end{equation}
Let $\varepsilon>0$. Then there exists $n_0(\varepsilon)$ such that 
for $n\geq n_0(\varepsilon)$
\begin{equation}\label{eq:liminf_estimation1}
k_1e^{-\alpha L(\tau_n)I(\tau_n)G(\tau_n)}-
k_2S(\tau_n)I(\tau_n)G(\tau_n)\geq k_1-\varepsilon.
\end{equation}
Since 
\begin{equation}
u(G(\tau_n))\leq s(G(\tau_n)-G_u)\leq sp-sG_u
\end{equation}
we get
\begin{equation}\label{eq:uG_lower_estimation}
-u(G(\tau_n))\geq sG_u-sp.
\end{equation}
Using \eqref{eq:liminf_estimation1} and \eqref{eq:uG_lower_estimation} we obtain
\begin{equation*}
0\geq k_1-\varepsilon+sG_u-sp.
\end{equation*}
Equivalently, 
\begin{equation}
p\geq G_u+\frac{k_1}{s}-\frac{\varepsilon}{s}.
\end{equation}
Since $\varepsilon>0$ is arbitrary we obtain
\begin{equation*}
p\geq G_u+\frac{k_1}{s}
\end{equation*}
which contradicts to \eqref{eq:p_less_than}. So 
\eqref{eq:liminf_lower_estimation} is proved.

Now we prove that
\begin{equation}\label{eq:liminf_is}
\liminf_{t\to\infty}G(t)= G_u+\frac{k_1}{s}.
\end{equation}
Indirectly, assume that
\begin{equation}
\liminf_{t\to\infty}G(t)> G_u+\frac{k_1}{s}.
\end{equation}
Let $\varepsilon>0$ such that 
\begin{equation*}
\liminf_{t\to\infty}G(t)>G_u+\frac{k_1}{s}+\varepsilon.
\end{equation*}
Then for $t\geq t_0(\varepsilon)$ 
\begin{equation*}
G(t)>G_u+\frac{k_1}{s}+\varepsilon,
\end{equation*}
so
\begin{equation}\label{eq:uG_lower}
u(G(t))>k_1+s\varepsilon.
\end{equation}
Since 
\begin{equation}\label{eq:k1}
k_1e^{-\alpha L(t)I(t)G(t)}-k_2S(t)I(t)G(t)<k_1,
\end{equation}
Using \eqref{eq:uG_lower} and \eqref{eq:k1}, from the equation of 
$G'(t)$ in \eqref{eq:non_auto_1} we get
\begin{equation*}
G'(t)< -s\varepsilon\qquad t\geq t_0(\varepsilon).
\end{equation*}
But this impossible because $G(t)\geq 0$. So \eqref{eq:liminf_is} 
is proved.

Lastly, we prove that
\begin{equation}\label{eq:limsup_is}
\limsup_{t\to\infty}G(t)= G_u+\frac{k_1}{s}.
\end{equation}
Indirectly, assume that 
\begin{equation}
G_u+\frac{k_1}{s}=\liminf_{t\to\infty}G(t)<\limsup_{t\to\infty}G(t).
\end{equation}
Then there exist a number $p$ such that
\begin{equation}
G_u+\frac{k_1}{s}=\liminf_{t\to\infty}G(t)<p<
\limsup_{t\to\infty}G(t).
\end{equation}
By the definition of $\limsup$ it follows that there exist bounded intervals 
$[a_n,b_n]$ such that $b_n<a_{n+1}$, 
$\displaystyle{\lim_{n\to\infty}}a_n=\infty$, $G(a_n),\,G(b_n)=p$ and 
by Rolle's mean value theorem there 
exist $\tau_n\in (a_n,b_n)$ such that $G'(\tau_n)=0$ and 
$G(\tau_n)>p$. 
Substituting $\tau_n$ into the equation of $G'(t)$ in 
\eqref{eq:non_auto_1} we get  
\begin{align}\label{eq:limsup_tau}
0&=k_1e^{-\alpha L(\tau_n)I(\tau_n)G(\tau_n)}-
k_2S(\tau_n)I(\tau_n)G(\tau_n)-u(G(\tau_n))\\
&< k_1-s(G(\tau_n)-G_u).\nonumber
\end{align}
From this inequality it follows
\begin{equation*}
G(\tau_n)<G_u+\frac{k_1}{s}.
\end{equation*}
But this is impossible because $G(\tau_n)>p>G_u+k_1/s$. So we proved 
 \eqref{eq:limsup_is}, and the statement of theorem is also proved. 
\end{proof}

\begin{theorem}\label{thm:Bt_est}
There exist constants $K_1,K_2>0$ such that 
\begin{equation*}
B(t)\leq K_1 e^{-K_2t}.
\end{equation*}
\end{theorem}

\begin{proof}
We know that (see \eqref{eq:B gen exp sol}) 
\begin{align}\label{eq:I1I2}
B(t) &=\left(
\int_0^t k_3 e^{-\lambda z +k_4 \int_0^z G(x)\,dx}
\,dz+
B_0
\right)e^{-k_4 \int_0^t G(z)\,dz} \nonumber \\
&=:B_1(t)+B_2(t).
\end{align}
First we estimate $B_1(t)$. Obviously
\begin{equation*}
\left(\int_0^t k_3 e^{-\lambda z +k_4 \int_0^z G(x)\,dx}
\,dz\right) e^{-k_4 \int_0^t G(z)\,dz}=\int_0^t k_3 e^{-\lambda z -k_4 \int_z^t G(x)\,dx}
\,dz.
\end{equation*}
Denote
\begin{equation*}
g_0:=\lim_{t\to\infty} G(t)=G_u+\frac{k_1}{s}.
\end{equation*}
Let $0<q<g_0$. Then there exists $T_q>0$ such that $G(x)> q$ for $x\geq T_q$. Thus for $t\geq T_q$ we 
have
\begin{align}\label{eq:J1J2}
B_1(t) &= \int_0^{T_q} k_3 e^{-\lambda z -k_4 \int_z^t G(x)\,dx}
\,dz+\int_{T_q}^t k_3 e^{-\lambda z -k_4 \int_z^t G(x)\,dx}\,dz \nonumber \\
&=:J_1(t)+J_2(t).
\end{align}
First we estimate $J_1(t)$. Since $z\leq T_q$ therefore
\begin{equation*}
\int_z^t G(x)\,dx=\int_z^{T_q} G(x)\,dx+\int_{T_q}^t G(x)\,dx.
\end{equation*}
When $x\geq T_q$ then $G(x)>q$ thus
\begin{equation*}
\int_{T_q}^t G(x)\,dx>q(t-T_q).
\end{equation*}
Hence we obtain 
\begin{equation*}
e^{-k_4\int_z^t G(x)\,dx}\leq e^{-k_4 q(t-T_q)} e^{-k_4\int_z^{T_q} G(x)\,dx}.
\end{equation*}
Thus we can estimate 
\begin{align}\label{eq:J1_est}
J_1(t) &\leq k_3 e^{-k_4 q(t-T_q)}\int_0^{T_q} e^{-\lambda z-k_4\int_z^{T_q} G(x)\,dx}\,dz\nonumber \\
&=\left(k_3 e^{k_4 qT_q } \int_0^{T_q} e^{-\lambda z -k_4 \int_z^{T_q} G(x)\,dx}\,dz \right)
e^{-k_4 q t} \nonumber\\
&=C_q e^{-k_4 q t}.
\end{align}
Now we estimate $J_2(t)$. If $x\geq T_q$ then $G(x)>q$. Thus
\begin{equation*}
J_2(t)<k_3 \int_{T_q}^t e^{-\lambda z -k_4 q(t-z)}\,dz=k_3 e^{-k_4 qt}\int_{T_q}^t e^{(k_4 q-\lambda)z}\,dz.
\end{equation*}
If $k_4 q>\lambda$ then 
\begin{equation}\label{eq:J2_first_est}
J_2(t)<\frac{k_3}{k_4q-\lambda} e^{-\lambda t}.
\end{equation}
If $k_4 q=\lambda$ then 
\begin{equation}\label{eq:J2_second_est}
J_2(t)<k_3 t e^{-\lambda t}.
\end{equation}
If $k_4 q<\lambda$ then 
\begin{equation}\label{eq:J2_third_est}
J_2(t)<k_3 \frac{e^{-\lambda T_q}}{\lambda - k_4 q}e^{-k_4 qt}.
\end{equation}
Finally we estimate $B_2(t)$. Assume that $t\geq T_q$. Then 
\begin{align}\label{eq:I2_est}
B_0 e^{-k_4 \int_0^t G(z)\,dz} &=e^{-k_4 \int_0^{T_q} G(z)\,dz} e^{-k_4 \int_{T_q}^t G(z)\,dz} \nonumber\\
&< B_0 e^{-k_4 \int_0^{T_q} G(z)\,dz} e^{-k_4 q(t-T_q)} \nonumber \\
&=B_0 e^{-k_4 \int_0^{T_q} G(z)\,dz +k_4 q T_q} e^{-k_4q t}\nonumber\\
&=C_1 e^{-k_4q t}.
\end{align}
From \eqref{eq:I1I2}-\eqref{eq:I2_est} the estimation in \autoref{thm:Bt_est} follows. 
\end{proof}

\begin{theorem}\label{thm:I_upper_estim}
There exist constants $K_3,K_4>0$ such that
\begin{equation*}
I(t)\leq K_3 e^{-K_4 t}.
\end{equation*}
\end{theorem}
\begin{proof}
We know that (see \eqref{eq:I gen exp sol})
\begin{equation}\label{eq:I0}
I(t) =\left(
\int_0^t k_6\frac{B(z)}{B_0}
\frac{(G(z))^{\gamma}}{c_1+(G(z))^{\gamma}}e^{k_5z}\,dz+
I_0
\right)e^{-k_5t}.
\end{equation}
Using \autoref{thm:G_upper_est} and \autoref{thm:Bt_est} we obtain 
\begin{equation}
\int_0^t k_6\frac{B(z)}{B_0}
\frac{(G(z))^{\gamma}}{c_1+(G(z))^{\gamma}}e^{k_5z}\,dz \leq \frac{k_6 K_1}{B_0}
\frac{G_{\max}^{\gamma}}{c_1+G_{\max}^{\gamma}}
\left( \int_0^t e^{-K_2 z+k_5 z}\,dz\right) e^{-k_5 t}. 
\end{equation}
If $k_5>K_2$ then
\begin{equation}
\left( \int_0^t e^{-K_2 z+k_5 z}\,dz\right) e^{-k_5 t}< e^{-K_2 t}.
\end{equation}
If $k_5=K_2$ then
\begin{equation}
\left( \int_0^t e^{-K_2 z+k_5 z}\,dz\right) e^{-k_5 t}=t e^{-k_5 t}.
\end{equation}
If $k_5<K_2$ then
\begin{equation}\label{eq:I_int_estimation}
\left( \int_0^t e^{-K_2 z+k_5 z}\,dz\right) e^{-k_5 t}<\frac{1}{K_2-k_5} e^{-k_5 t}.
\end{equation}
From \eqref{eq:I0}-\eqref{eq:I_int_estimation} the estimation in \autoref{thm:I_upper_estim} follows.
\end{proof}
\begin{theorem}
There exist constants $K_5,K_6>0$ such that 
\begin{equation*}
\left|  G(t)-\left( G_u+\frac{k_1}{s} \right) \right|\leq K_5 e^{-K6 t}.
\end{equation*}
\end{theorem}
\begin{proof}

We know (\autoref{thm:lim_G}) that there exists $t_0>0$ such that $G(t)>G_u$. Thus for $t>t_0$ 
the equation \eqref{eq:non_auto_1} becomes 
\begin{equation}\label{eq:G_prime}
G'(t)=k_1 e^{-\alpha_0 L(t)I(t) G(t)}-k_2 S(t)I(t)G(t)-s(G(t)-G_u).
\end{equation}
Obviously
\begin{equation}
k_1 e^{-\alpha_0 L(t)I(t) G(t)}=k_1+O\left( L(t)I(t)G(t) \right).
\end{equation}
Since $G(t)$ is bounded, using \autoref{thm:L_t_estimation}, \autoref{thm:S_t_estimation}, and 
\autoref{thm:I_upper_estim} we have
\begin{equation}
k_1 e^{-\alpha_0 L(t)I(t) G(t)}=k_1+O\left( e^{-K_4 t}\right).
\end{equation}
Similarly
\begin{equation}
k_2 S(t)I(t)G(t)=O\left( e^{-K_4 t}\right).
\end{equation}
Thus from equation \eqref{eq:G_prime} it follows
\begin{equation}
G'(t)=-s\left(G(t)- \left( G_u+\frac{k_1}{s} \right) \right)+h(t),
\end{equation}
where $h(t)=O\left( e^{-K_4 t}\right)$.
Now we introduce the function 
\begin{equation}
H(t):=G(t)-\left( G_u+\frac{k_1}{s} \right).
\end{equation}
Then we get
\begin{equation}
H'(t)=-sH(t)+h(t). 
\end{equation}
From Gronwall inequality we obtain
\begin{align*}
H(t) &\leq \left(H(t_0)e^{st_0}  \right) e^{-st}+\int_{t_0}^t e^{-s(t-z)} h(z)\,dz \\
&\leq\left(H(t_0)e^{st_0}\right)e^{-st}+c_0\int_{t_0}^t e^{-s(t-z)} e^{-K_4 z} \,dz \\
&=\left(H(t_0)e^{st_0}\right)e^{-st}+c_0\int_{t_0}^t e^{-K_4 z+sz}  \,dz\, e^{-st}.
\end{align*}
Applying the method was used in \autoref{thm:I_upper_estim} we get the upper estimation. 
Using the reverse Gronwall inequality, we get the lower estimation. 
\end{proof}

\subsection{\texorpdfstring{The system $A(t)=A_0$}{The system A(t)=A0}}

\begin{theorem}
For $t\geq 0$ we have
\begin{equation*}
L(t) \geq \frac{k_{12}+k_{14}c_2}{k_{15}c_5+k_{13}}+
\left( L_0-\frac{k_{12}+k_{14}c_2}{k_{15}c_5+k_{13}} \right)e^{-(k_{15}c_5+k_{13})t},
\end{equation*}
and
\begin{equation*}
L(t) \leq \frac{k_{12}+k_{14}c_3}{k_{15}c_4+k_{13}}+
\left( L_0-\frac{k_{12}+k_{14}c_3}{k_{15}c_4+k_{13}} \right)e^{-(k_{15}c_4+k_{13})t}.
\end{equation*}
\end{theorem}
\begin{proof}
We know that (\eqref{eq:L gen exp sol A0})
\begin{align*}
L(t) = & \left(\int_0^t \left( k_{12} + k_{14}Y(z) \right) 
e^{\int_0^z k_{15}F(x)+k_{13} \, dx} \, dz +L_0 \right) \\
   &\times e^{-\int_0^t k_{15}F(x)+k_{13} \, dx}. 
\end{align*}
Since $c_4\leq F(t)\leq c_5$ we obtain the following estimations
\begin{equation}
 L_0 e^{-\int_0^t k_{15}F(x)+k_{13} \, dx}\geq L_0 e^{-(k_{15}c_5+k_{13})t},
\end{equation}
and 
\begin{equation}
 L_0 e^{-\int_0^t k_{15}F(x)+k_{13} \, dx}\leq L_0 e^{-(k_{15}c_4+k_{13})t}.
\end{equation}
The other term can be written in the following form
\begin{equation}
\int_0^t \left( k_{12} + k_{14}Y(z) \right) 
e^{-\int_z^t k_{15}F(x)+k_{13} \, dx} \, dz. 
\end{equation}
Since  $c_2\leq Y(t)\leq c_3$,  we can estimate this integral 
\begin{align}
\int_0^t \left( k_{12} + k_{14}Y(z) \right) 
e^{-\int_z^t k_{15}F(x)+k_{13} \, dx} \, dz &\geq (k_{12}+k_{14}c_2) 
\int_0^t e^{-(k_{15}c_5+k_{13})(t-z)}\,dz \nonumber\\
&=\frac{k_{12}+k_{14}c_2}{k_{15}c_5+k_{13}}\left(1-e^{-(k_{15}c_5+k_{13})t}  \right),
\end{align}
and
\begin{align}
\int_0^t \left( k_{12} + k_{14}Y(z) \right) 
e^{-\int_z^t k_{15}F(x)+k_{13} \, dx} \, dz &\leq (k_{12}+k_{14}c_3) 
\int_0^t e^{-(k_{15}c_4+k_{13})(t-z)}\,dz \nonumber\\
&=\frac{k_{12}+k_{14}c_3}{k_{15}c_4+k_{13}}\left(1-e^{-(k_{15}c_4+k_{13})t}  \right).
\end{align}
\end{proof}

\begin{corollary}
If $Y(t)=Y_0$, $F(t)=F_0$, then for $t\geq 0$ we have
\begin{equation*}
L(t)=\frac{k_{12}+k_{14}Y_0}{k_{15}F_0+k_{13}}+\left( L_0-\frac{k_{12}+k_{14}Y_0}{k_{15}F_0+k_{13}} 
\right) e^{-(k_{15}F_0+k_{13})t}.
\end{equation*}
\end{corollary}

In the same way we obtain
\begin{theorem}
For $t\geq 0$ we have
\begin{equation*}
S(t) \geq \frac{k_{7}+k_{9}c_2}{k_{10}c_5+k_{8}}+
\left( L_0-\frac{k_{7}+k_{9}c_2}{k_{10}c_5+k_{8}} \right)e^{-(k_{10}c_5+k_{8})t},
\end{equation*}
and
\begin{equation*}
S(t) \leq \frac{k_{7}+k_{9}c_3}{k_{10}c_4+k_{8}}+
\left( L_0-\frac{k_{7}+k_{9}c_3}{k_{10}c_4+k_{8}} \right)e^{-(k_{10}c_4+k_{8})t}.
\end{equation*}
\end{theorem}

\begin{corollary}
If $Y(t)=Y_0$, $F(t)=F_0$, then for $t\geq 0$ we have
\begin{equation*}
S(t)=\frac{k_{7}+k_{9}Y_0}{k_{10}F_0+k_{8}}+\left( S_0-\frac{k_{7}+k_{9}Y_0}{k_{10}F_0+k_{8}} 
\right) e^{-(k_{15}F_0+k_{13})t}.
\end{equation*}
\end{corollary}

The biological meaning of $G(t)$, and that the system \eqref{eq:1}-\eqref{eq:6} in a very special case has 
a unique positive equilibrium point that is locally asymptotically stable for all positive values 
of parameters (\cite{degaetano2024} Proposition 7.) motivate the following 

\begin{openproblem}\label{open:G bounded below}
Does exist a number $c_0>0$ such that $G(t)> c_0$ for all $t\geq t_{c_0}$?
\end{openproblem}

In the following we discuss what are the consequences of the 'yes' answer. 

\begin{theorem}\label{thm:B_upper_est_A0}
There exists constant $K_6>0$ such that
\begin{equation*}
B(t)\leq K_6.
\end{equation*}
\end{theorem}
\begin{proof}
We know that (see \eqref{eq:B gen exp sol A0}) 
\begin{align}\label{eq:B1B2_A0}
B(t) &=\left(
\int_0^t k_3 e^{k_4 \int_0^z G(x)\,dx}
\,dz+B_0\right)e^{-k_4 \int_0^t G(z)\,dz} \nonumber\\
&=:B_1(t)+B_2(t). 
\end{align}
We adopt the method was used in \autoref{thm:Bt_est}. 
First we estimate $B_1(t)$. Obviously 
\begin{equation*}
\left(\int_0^t k_3 e^{k_4 \int_0^z G(x)\,dx}
\,dz\right) e^{-k_4 \int_0^t G(z)\,dz}=\int_0^t k_3 e^{ -k_4 \int_z^t G(x)\,dx}
\,dz.
\end{equation*}
By the hypothetical assumption we have $G(t)> c_0$ for $t\geq t_{c_0}$. Thus for $t\geq t_{c_0}$ 
we have
\begin{align}\label{eq:J1J2_A0}
B_1(t) &= \int_0^{t_{c_0}} k_3 e^{-k_4 \int_z^t G(x)\,dx}
\,dz+\int_{t_{c_0}}^t k_3 e^{ -k_4 \int_z^t G(x)\,dx}\,dz \nonumber \\
&=:J_1(t)+J_2(t).
\end{align}
First we estimate $J_1(t)$. Since $z\leq t_{c_0}$ therefore
\begin{equation*}
\int_z^t G(x)\,dx=\int_z^{t_{c_0}} G(x)\,dx+\int_{t_{c_0}}^t G(x)\,dx.
\end{equation*}
When $x\geq t_{c_0}$ then $G(x)> c_0$ thus
\begin{equation*}
\int_{t_{c_0}}^t G(x)\,dx>c_0 (t-t_{c_0}).
\end{equation*}
Hence we obtain
\begin{equation*}
e^{-k_4\int_z^t G(x)\,dx}\,dz \leq e^{-k_4 c_0(t-t_{c_0})} e^{-k_4 \int_z^{t_{c_0}} G(x)\,dx}.
\end{equation*}
Thus we can estimate
\begin{align}
J_1(t) &\leq k_3 e^{-k_4 c_0(t-t_{c_0})} \int_0^{t_{c_0}}e^{-k_4\int_z^{t_{c_0}} G(x)\,dx}\,dz\nonumber\\
&=\left(  k_3 e^{k_4 c_0 t_{c_0}} \int_0^{t_{c_0}}e^{-k_4\int_z^{t_{c_0}} G(x)\,dx}\,dz \right)
e^{-k_4c_0 t}\nonumber\\
&=C_{c_0}e^{-k_4 c_0 t}.
\end{align}
Now we estimate $J_2(t)$. If $x\geq t_{c_0}$ then $G(x)>c_0$. Thus we get 
\begin{equation}
J_2(t)< k_3 \int_{t_{c_0}}^t e^{-k_4c_0(t-z)}\,dz\leq K.
\end{equation}
Lastly, similarly as in \autoref{thm:Bt_est}, we obtain. 
\begin{equation}\label{eq:B2_upper_A0}
    B_2(t)\leq C_2 e^{-k_4 c_0 t}.
\end{equation}
From \eqref{eq:B1B2_A0}-\eqref{eq:B2_upper_A0} the theorem follows.
\end{proof}

\begin{theorem}
There exists constant $K_7>0$ such that 
\begin{equation*}
I(t)\leq K_7.
\end{equation*}
\end{theorem}
\begin{proof}
We know (see \eqref{eq:I gen exp sol})
\begin{equation}
I(t)=\left(
\int_0^t k_6\frac{B(z)}{B_0}
\frac{(G(z))^{\gamma}}{c_1+(G(z))^{\gamma}}e^{k_5z}\,dz+
I_0 \right)e^{-k_5t}
\end{equation}
Using \autoref{thm:G_upper_est} and \autoref{thm:B_upper_est_A0} the theorem follows.
\end{proof}

\subsubsection{Local stability of the system}

In \cite{degaetano2024} Appendix B the local stability for 
fixed age ($A(t)=A_0$) was investigated with the assumptions $F(t)=F_0$, $Y(t)=Y_0$. 
Since we use $u(G)$ instead of $R(G)$ we have to modify this investigation, although we use the same 
standard method as in the Appendix B. First, we rewrite the system with a simplified notation as 
\begin{align}
G' &=m_1 e^{-\alpha_0 LIG}-m_2 SIG-u(G), \label{eq:loc_stab_Gprime}\\
B' &=m_3 -m_4 GB, \label{eq:loc_stab_Bprime}\\
I' &=-m_5 I +m_6 B\cdot\frac{G^{\gamma}}{c_1+G^{\gamma}}, \label{eq:loc_stab_Iprime}\\
S' &=m_7-m_8 S \label{eq:loc_stab_Sprime}\\
L' &=m_9-m_{10}L \label{eq:loc_stab_Lprime},
\end{align}
where $m_1=k_1$, $m_2=k_2$, $m_3=k_3$, $m_4=k_4$, $m_5=k_5$, $m_6=k_6/B_0$, $m_7=k_7+k_9 Y_0$, 
$m_8=k_8+k_{10}F_0$, $m_9=k_{12}+k_{14}Y_0$, $m_{10}=k_{13}+k_{15}F_0$. 

\begin{theorem}
System \eqref{eq:loc_stab_Gprime}-\eqref{eq:loc_stab_Lprime} has a unique positive equilibrium point. 
\end{theorem}
\begin{proof}
Solving the equilibrium conditions $B'=0$, $I'=0$, $S'=0$, $L'=0$ we obtain that 
\begin{equation}\label{eq:equi_value}
B_{*}=\frac{m_3}{m_4 G},\quad I_{*}=\frac{m_3 m_6}{m_4 m_5}\frac{G^{\gamma -1}}{c_1+G^{\gamma}},
\quad S_{*}=\frac{m_7}{m_8}, \quad L_{*}=\frac{m_9}{m_{10}}.
\end{equation}
Here the equilibrium value $B_{*}$ and $I_{*}$ depends uniquely on $G$, while 
$S_{*}$ and $L_{*}$ are positive 
constants. So we have to show that the equilibrium value of $G$ is unique. Substituting these values 
into the condition $G'=0$ in \eqref{eq:loc_stab_Gprime} we obtain that 
\begin{equation}
m_1 \exp\left(-\alpha_0\frac{m_3 m_6 m_9}{m_4 m_5 m_{10}}\frac{G^{\gamma}}{c_1+G^{\gamma}} \right) 
-\frac{m_2 m_3 m_6 m_7}{m_4 m_5 m_8}\frac{G^{\gamma}}{c_1+G^{\gamma}}-u(G)=0.
\end{equation}
Denote
\begin{equation}
\alpha:=\alpha_0\frac{m_3 m_6 m_9}{m_4 m_5 m_{10}}>0,\quad \beta:=\frac{m_2 m_3 m_6 m_7}{m_4 m_5 m_8}>0.
\end{equation}
Let us introduce the continuous function
\begin{equation}\label{eq:h_function}
h(G):=\beta\frac{G^{\gamma}}{c_1+G^{\gamma}}+u(G)-m_1\exp\left(-\alpha \frac{G^{\gamma}}{c_1+G^{\gamma}}  \right)
\end{equation}
defined for $G\geq 0$. Then $h$ has the following properties:

(i) $h(0)=-m_1<0$.

(ii) $h(G)\to\infty$ as $G\to\infty$.

We show that the function $h$ increases strictly monotonically.
We have to distinguish two cases. 
(a) $0< G< G_u$ and (b) $G> G_u$.

(a) In this case 
\begin{equation}
h(G)=\beta\frac{G^{\gamma}}{c_1+G^{\gamma}}-m_1\exp\left(-\alpha \frac{G^{\gamma}}{c_1+G^{\gamma}}  \right)
\end{equation}
and
\begin{equation}
h'(G) = \frac{\gamma c_1 G^{\gamma-1}}{(c_1 + G^{\gamma})^2} \left[ \beta + m_1 \alpha \exp\left(-\alpha \frac{G^{\gamma}}{c_1 + G^{\gamma}}\right) \right]>0.
\end{equation}

(b) In this case 
\begin{equation}
h(G)=\beta\frac{G^{\gamma}}{c_1+G^{\gamma}}+s(G-G_u) 
-m_1\exp\left(-\alpha \frac{G^{\gamma}}{c_1+G^{\gamma}}  \right)
\end{equation}
and
\begin{equation}
h'(G) = \frac{\gamma c_1 G^{\gamma-1}}{(c_1 + G^{\gamma})^2} \left[ \beta + m_1 \alpha \exp\left(-\alpha \frac{G^{\gamma}}{c_1 + G^{\gamma}}\right) \right]+s>0.
\end{equation}

Therefore, $h(G)$ has a unique positive equilibrium point.
\end{proof}

\begin{theorem}
The equilibrium point is locally asymptotically stable for all positive values of the parameters. 
\end{theorem}
\begin{proof}
Let $G_{*}, B_{*}, I_{*}, S_{*}, L_{*}$ denote the coordinates of the unique positive equilibrium point. 
We have to distinguish three cases. 

Case 1. $0<G_{*}<G_u$. 

Since $h$ increases strictly monotonically, the necessary and sufficient assumption is $h(G_u)>0$, that is,
\begin{equation}
\beta\frac{G_u^{\gamma}}{c_1+G_u^{\gamma}}>m_1\exp\left(-\alpha \frac{G_u^{\gamma}}{c_1+G_u^{\gamma}}  \right).
\end{equation}
The Jacobian of the system is 
\begin{equation}
J(G_{*}, B_{*}, I_{*}, S_{*}, L_{*}) = 
\begin{bmatrix} 
-b_1 & 0 & -b_2 & -b_3 & -b_4 \\
-b_5 & -b_6 & 0 & 0 & 0 \\
d_1 & d_2 & -m_5 & 0 & 0 \\
0 & 0 & 0 & -m_8 & 0 \\
0 & 0 & 0 & 0 & -m_{10} 
\end{bmatrix}
\end{equation}
where 
\begin{align}
b_1 &=\alpha_0 L_{*} I_{*} m_1 e^{-\alpha_0 L_{*} I_{*} G_{*}} + m_2 S_{*} I_{*} \label{eq:J_entries_b1} \\
b_2 &=\alpha_0 L_{*} G_{*} m_1 e^{-\alpha_0 L_{*} I_{*} G_{*}} + m_2 S_{*} G_{*} \\
b_3 &=m_2 I_{*} G_{*} \\
b_4 &= \alpha_0 I_{*} G_{*} m_1 e^{-\alpha_0 L_{*} I_{*} G_{*}} \\
b_5 &= m_4 B_{*} \\
b_6 &= m_4 G_{*} \\
d_1 &= \frac{m_6 \gamma c_1 G_{*}^{\gamma-1} B_{*}}{(c_1 + G_{*}^{\gamma})^2} \\
d_2 &= \frac{m_6 G_{*}^{\gamma}}{c_1 + G_{*}^{\gamma}} \label{eq:J_entries_d2}
\end{align}
The characteristic polynomial of $J$ is 
\begin{equation}
p(\lambda)=(\lambda +m_8)(\lambda +m_{10})(\lambda^3+a_2\lambda^2+a_1\lambda+a_0),
\end{equation}
where
\begin{align*}
a_2 &=m_5+b_6+b_1 \\
a_1 &=(b_1+m_5)b_6+m_5b_1+d_1b_2 \\
a_0 &=(m_5b_1+d_1b_2)b_6-d_2b_5b_2.
\end{align*}
The equilibrium point of the system is locally stable if all the roots of $p(\lambda)$ have negative real parts. 
Since $p(\lambda)$ has two negative roots, $-m_8$ and $-m_{10}$, we have to investigate the roots of 
$\lambda^3+a_2\lambda^2+a_1\lambda+a_0$. By the Routh-Hurwitz stability criterion this third-degree polynomial has all 
roots with negative real parts if and only if 
\begin{equation}
a_2, a_1, a_0>0,\, a_1a_2-a_0>0. 
\end{equation}
Obviously $a_2,a_1>0$. Furthermore
\begin{equation}
a_1a_2-a_0=(m_5+b_6)b_1^2+(d_1b_2+(b_6+m_5)^2)b_1+b_6m_5^2+(d_1b_2+b_6^2)m_5+d_2b_5b_2>0.
\end{equation}
Now we calculate the value of $a_0$. Denote 
\begin{equation}
E_{*}:=\exp\left( -\alpha_0 L_{*} I_{*} G_{*} \right).
\end{equation}
First we calculate $b_1b_6m_5-b_2b_5d_2$. Obviously
\begin{equation}
b_1=\frac{b_2}{G_{*}}I_{*}.
\end{equation}
Using $b_5=m_4 B_{*}$ and $b_6=m_4 G_{*}$ it follows 
\begin{align}
b_1b_6m_5-b_2b_5d_2 &=\left(\frac{b_2}{G_{*}}I_{*} \right)(m_4 G_{*})m_5-b_2(m_4B_{*})d_2 \nonumber \\
&=I_{*}b_2m_4m_5-b_2m_4B_{*}d_2 \nonumber\\
&=b_2m_4(I_{*}m_5-B_{*}d_2) \nonumber\\
&=b_2m_4\left[\left( \frac{m_3m_6}{m_4m_5}\frac{G_{*}^{\gamma-1}}{c_1+G_{*}^{\gamma}} \right)m_5 -
\left( \frac{m_3}{m_4G_{*}}\right) \left( m_6\frac{G_{*}^{\gamma}}{c_1+G_{*}^{\gamma}}  \right)   \right] \nonumber\\
&=0. \label{eq:a0_zero}
\end{align}
Thus we get
\begin{equation}
a_0=\frac{\gamma c_1 m_3 m_6 G_{*}^{\gamma}}{m_{10}m_8 (c_1+G_{*}^{\gamma})^2}(\alpha_0 m_1m_8m_9E_{*}+m_2m_7m_{10})>0.
\end{equation}
Thus we have proved that the equilibrium point of the system is locally stable.

Case 2. $G_{*}>G_u$. 

Since $h$ increases strictly monotonically, the necessary and sufficient assumption is $h(G_u)<0$, that is,
\begin{equation}
\beta\frac{G_u^{\gamma}}{c_1+G_u^{\gamma}}<m_1\exp\left(-\alpha \frac{G_u^{\gamma}}{c_1+G_u^{\gamma}}  \right).
\end{equation}
We follow the calculation of Case 1. In this case 
\begin{equation*}
b_1 =\alpha_0 L_{*} I_{*} m_1 e^{-\alpha_0 L_{*} I_{*} G_{*}} + m_2 S_{*} I_{*}+s
\end{equation*}
and 
\begin{equation*}
b_1b_6m_5-b_2b_5d_2=sm_4m_5G_{*}.
\end{equation*}
Finally,
\begin{equation}
a_0=\frac{\gamma c_1 m_3 m_6 G_{*}^{\gamma}}{m_{10}m_8 (c_1+G_{*}^{\gamma})^2}(\alpha_0 m_1m_8m_9E_{*}+m_2m_7m_{10})+sm_4m_5G_{*}>0.
\end{equation}
So in this case the equilibrium point of the system is locally stable.

Case 3. $G_{*}=G_u$. 

We will investigate the sign of $G'$ around $G_u$. 
Let us substitute 
\begin{align*}
L_{*} &= \frac{m_9}{m_{10}}, \\
I_{*} &= \frac{m_3m_6}{m_4m_5}\frac{G_u^{\gamma -1}}{c_1+G_u^{\gamma}}, \\
S_{*} &= \frac{m_7}{m_8}
\end{align*}
into the right-hand-side of \eqref{eq:loc_stab_Gprime}. 
Then we obtain
\begin{align*}
G' &=m_1 e^{-\alpha_0 L_{*}I_{*}G} -m_2 S_{*}I_{*}G-u(G) \\
&=:f(G).
\end{align*}
Then $f(0)=m_1>0$, $f(G_u)=0$, and $f(G)\to -\infty$ as $G\to\infty$. When $0<G<G_u$ 
then $f(G)=m_1 \exp\left(-\alpha_0 L_{*}I_{*}G\right) -m_2 S_{*}I_{*}G$, 
and $f'(G)=-m_1 \alpha_0 L_{*}I_{*} \exp\left(-\alpha_0 L_{*}I_{*}G\right) -m_2 S_{*}I_{*}<0$. 
Thus $f$ is strictly monotonically decreasing, that implies $f(G)>0$ on $[0,G_u)$. 
Hence $G'>0$, the $G$ coordinate of the 
trajectory approaches the boundary surface ($G=G_u$) from below. Similarly, 
when $G>G_u$ then $f(G)=m_1 \exp\left(-\alpha_0 L_{*}I_{*}G\right) -m_2 S_{*}I_{*}G-s(G-G_u)$, and 
$f'(G)=-m_1 \alpha_0 L_{*}I_{*} \exp\left(-\alpha_0 L_{*}I_{*}G\right) -m_2 S_{*}I_{*}-s<0$.
Thus $f$ is strictly monotonically decreasing, that implies $f(G)<0$ on $(G_u,\infty)$. 
Hence $G'<0$, the $G$ coordinate of the trajectory approaches the boundary surface ($G=G_u$) from above.

Similarly, if $S<S_{*}$ then $S'>0$, the $S$ trajectory approaches $S_{*}$ from left, 
and if $S>S_{*}$ then $S'>0$, the $S$ coordinate of the trajectory approaches $S_{*}$ from right. 
In the same way we obtain that the $L$, $I$, $B$ coordinates of the trajectory approach 
$L_{*}$, $T_{*}$, $B_{*}$ respectively.   It follows that the equilibrium point of the system is locally stable. 
\end{proof}

\section*{Acknowledgements}

The author of this paper would like to thank I. Nagy (one of the authors of \cite{degaetano2024}) 
for the detailed discussion regarding the interpretation of the model.

\section*{Declarations}

\textbf{Conflict of interest}  The author declares that he has no conflict of interest.
%%%%%%%%%%%%%%%%%%%%%%%%%%%%%%%%%%%%%%%%%%%%%%%%%%%%%%%%%%%%%%%%%%%%%%%%%%%%%%%
%\bibliography{sn-bibliography}

\end{document}